\documentclass[12pt,a4paper,UKenglish]{article}
\usepackage[margin=2.5cm]{geometry}
\usepackage{amsthm} 
\usepackage{amssymb}
\usepackage{amsopn}
\usepackage{color}
\usepackage{babel} 
\usepackage{hyperref}
\usepackage{amsmath, amsfonts, amssymb} 
\usepackage{graphicx}
\allowdisplaybreaks 
\usepackage[utf8]{inputenc}
\usepackage[numbers,sort&compress]{natbib}

\usepackage{esvect}

\usepackage{tikz}
\usetikzlibrary{topaths,calc}

\usepackage{amssymb, amsmath}
\usepackage{verbatim}
\usepackage{hyperref}

\theoremstyle{plain}
\newtheorem{theorem}{Theorem}[section]

\newtheorem{corollary}[theorem]{Corollary}

\theoremstyle{definition}
\newtheorem{definition}[theorem]{Definition}
\newtheorem{example}[theorem]{Example}

\theoremstyle{remark}

\newtheorem{remark}[theorem]{Remark}

\DeclareMathOperator{\RQ}{RQ}
\DeclareMathOperator{\vol}{vol}
\DeclareMathOperator{\id}{Id}
\DeclareMathOperator{\diag}{diag}

\usepackage[affil-it]{authblk}
\usepackage{comment}
\begin{document}

\title{A Cheeger Cut for Uniform Hypergraphs}

\author[1,2]{Raffaella Mulas\footnote{Email address: r.mulas@soton.ac.uk\\
This work was supported by The Alan Turing Institute under the EPSRC grant EP/N510129/1.}}
\affil[1]{The Alan Turing Institute, London, UK}
\affil[2]{University of Southampton, Southampton, UK}

\date{}
	
\maketitle

\begin{abstract}
The graph Cheeger constant and Cheeger inequalities are generalized to the case of hypergraphs whose edges have the same cardinality. In particular, it is shown that the second largest eigenvalue of the generalized normalized Laplacian is bounded both above and below by the generalized Cheeger constant, and the corresponding eigenfunctions can be used to approximate the Cheeger cut.\newline
\vspace{0.2cm}

\end{abstract}

\section{Introduction}
\begin{figure}[t]
    \centering
    \includegraphics[width=7cm]{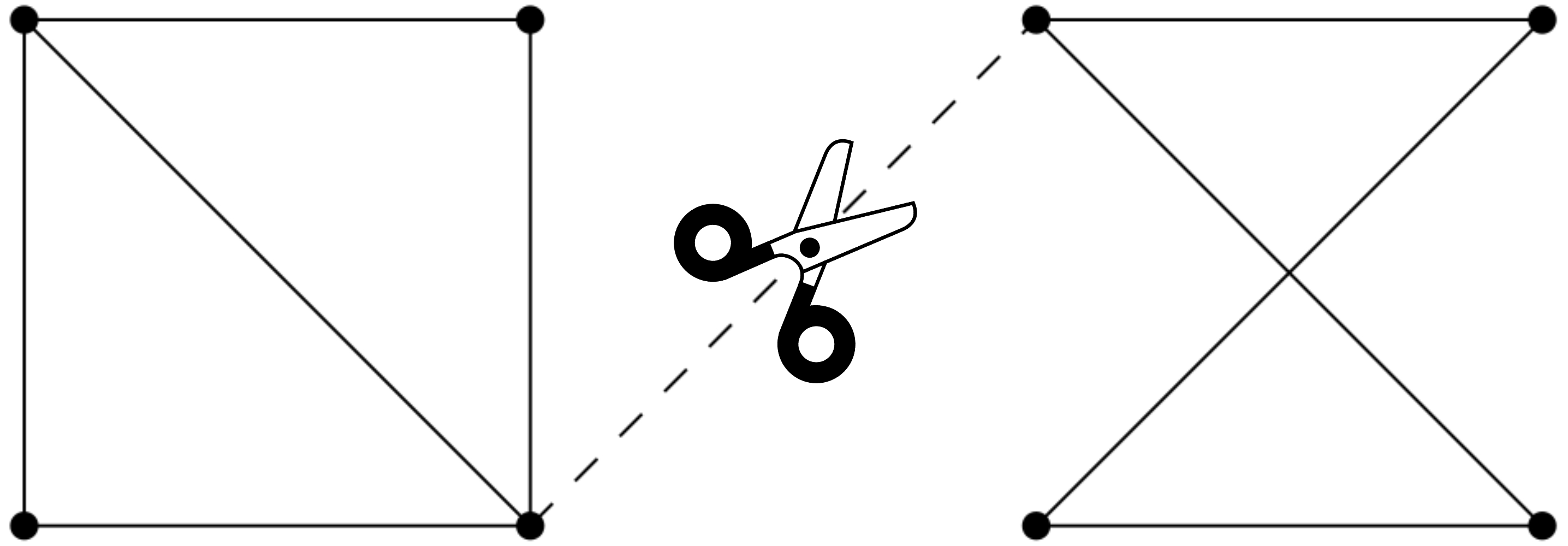}
    \caption{The Cheeger cut on a graph.}
    \label{fig:graph}
\end{figure}
\textbf{Historical note. }Cheeger constants and Cheeger inequalities have a long history. The now-called \textsl{Cheeger constant} of a simple graph $G=(V,E)$ was introduced in 1951 by George Pólya and Gábor Szegő \cite{Polya}, who called it the \textsl{isoperimetric constant} and defined it as
\begin{equation*}
    h(G):=\min_{\emptyset\neq S\subsetneq V}\frac{|E(S,\bar{S})|}{\min\{\vol(S),\vol(\bar{S})\}},
\end{equation*}where $E(S,\bar{S})$ denotes the set of edges between $S$ and its complement $\bar{S}:=V\setminus S$, while the \textsl{volume} of $S$, denoted $\vol(S)$, is the sum of the vertex degrees in $S$. Finding a set $S$ realizing the Cheeger constant means finding a small \textsl{edge cut} $E(S,\bar{S})$ such that, if removed from $G$, it divides the graph into two disconnected components that have roughly equal volume (Figure \ref{fig:graph}). Therefore, $h$ measures how different $G$ is from a disconnected graph, and it is largest for the complete graph.\newline
The continuous analogue of $h(G)$ was then defined by Jeff Cheeger \cite{CheegerPhD} in 1970, in the context of spectral geometry, as follows. Given a compact $n$-dimensional manifold $M$, let
\begin{equation*}
    h(M):=\inf_D\frac{\vol_{n-1}(\delta D)}{\vol_n(D)},
\end{equation*}where $D\subset M$ is a smooth $n$-submanifold with boundary $\delta D$ and $0<\vol_n(D)\leq\vol(M)/2$. Cheeger proved that the first nonvanishing eigenvalue $\lambda_{\min}(M)$ of the Laplace-Beltrami operator is such that
\begin{equation*}
    \lambda_{\min}(M)\geq \frac{1}{4}h^2(M)
\end{equation*}and, as shown by Peter Buser \cite{Buser1} in 1978, for each compact manifold there exist Riemannian metrics for which the inequality becomes sharp. In a later work in 1982, Buser \cite{Buser2} also proved that, if the Ricci curvature of a compact unbordered Riemannian $n$-manifold $M$ is bounded below by $-(n-1)a^2$, for some $a\geq 0$, then
\begin{equation*}
    \lambda_{\min}(M)\leq 2a(n-1)h+10h^2.
\end{equation*}Therefore, $h(M)$ can be used to estimate $\lambda_{\min}(M)$ and vice versa.\newline
In 1984-5, Jozef Dodziuk \cite{dodziuk} and Noga Alon with Vitali Milman \cite{Alon} derived analogous estimates for the graph Cheeger constant and for the first nonvanishing eigenvalue of the Kirchhoff Laplacian associated to a connected graph. Similarly, in 1992, Fan Chung \cite{Chung} proved the Cheeger inequalities for the \textsl{symmetric normalized Laplacian} of a graph $G$ on $n$ nodes, that she defined as
\begin{equation*}
   \mathcal{L}(G):=\id-D(G)^{-1/2}A(G)D(G)^{-1/2},
\end{equation*}where $\id$ is the $n\times n$ identity matrix, $D(G)$ is the diagonal degree matrix and $A(G)$ is the adjacency matrix of $G$. Chung proved that $\mathcal{L}(G)$ has $n$ real, nonnegative eigenvalues, denoted $\lambda_1(G)\leq\ldots\leq \lambda_n(G)$, that encode many qualitative properties of $G$. In particular, she proved that, for a connected graph, the first two eigenvalues are such that $\lambda_1(G)=0$ and
\begin{equation}\label{eq:Cheeger1}
    \frac{1}{2}h(G)^2\leq \lambda_2(G)\leq 2h(G).
\end{equation}Therefore, as well as in the continuous case, $h(G)$ can be used to estimate $\lambda_2(G)$ and vice versa. Moreover, the eigenvectors corresponding to $\lambda_2(G)$ can be used in order to approximate the Cheeger cut, as follows. An eigenvector for $\mathcal{L}(G)$ can be seen as a function $f:V\rightarrow \mathbb{R}$ and, if $f$ is an eigenfunction with eigenvalue $\lambda_2(G)$, then $f$ must achieve both positive and negative values, and the edges between the sets
\begin{equation*}
    \{v\in V:f(v)\geq 0\} \quad \text{and} \quad \{v\in V:f(v)<0\}
\end{equation*}approximate the Cheeger cut. Since solving the Cheeger cut problem is NP-hard \cite{np}, while the eigenvalues and the eigenvectors of $\mathcal{L}(G)$ can be found quickly, spectral clustering based on these results is a very common tool and have found many applications, see for instance \cite{appl2007,appl2008,appl2012,appl2015}. Citing \cite{Clustering} (Ulrike von Luxburg, 2007), ``\textsl{In  recent  years,  spectral  clustering  has  become  one  of  the  most  popular  modern  clustering algorithms.  It is simple to implement, can be solved efficiently by standard linear algebra software, and very often outperforms traditional clustering algorithms such as the k-means algorithm}''.\newline
Note that the \textsl{normalized Laplacian} or \textsl{random walk Laplacian}
\begin{equation*}
    L(G):=\id-D(G)^{-1}A(G)=D(G)^{-1/2}\mathcal{L}(G)D(G)^{1/2}
\end{equation*}is similar to $\mathcal{L}(G)$, therefore these two matrices have the same spectrum. Moreover, $f$ is an eigenfunction for $\mathcal{L}(G)$ with eigenvalue $\lambda$ if and only if $D^{1/2}f$ is an eigenfunction for $L(G)$ with eigenvalue $\lambda$. Hence, the above statements for $\mathcal{L}(G)$ can be equivalently stated for $L(G)$, on which we will focus throughout this paper. \newline

\textbf{Aim of this work.} The aim of this work is to generalize the graph Cheeger inequalities and Cheeger cut to the case of \textsl{uniform hypergraphs}. Hypergraphs are a generalization of graphs in which vertices are joined by sets of any cardinality, and a hypergraph is said to be $k$-\textsl{uniform} if all its edges have cardinality $k$. Hypergraphs find applications in many real networks (e.g.\ cellular networks~\cite{KlamtHausTheis}, social networks,~\cite{ZhangLiu}, neural networks~\cite{neuro1}, opinion formation~\cite{LanchierNeufer}, epidemic networks~\cite{BodoKatonaSimon}) and a hypergraph Cheeger cut could be applied to clustering problems on such networks.\newline
The fundamental idea used here is the following. Given a connected simple graph $G$, its \textsl{signless normalized Laplacian} is
\begin{equation*}
    L^+(G):=\id+D(G)^{-1}A(G)=2\id -L(G).
\end{equation*}It is such that
\begin{equation*}
    \lambda \text{ is an eigenvalue for }L(G) \iff 2-\lambda \text{ is an eigenvalue for }L^+(G)
\end{equation*}with the same eigenfunctions and, moreover, $L^+(G)=L(G^+)$, where $G^+$ is the \textsl{signed graph} obtained from $G$ by letting each edge have a positive sign. Since, furthermore, $h(G)=h(G^+)$, the Cheeger inequalities in \eqref{eq:Cheeger1} can be equivalently reformulated in terms of the second largest eigenvalue of $L(G^+)$, as
\begin{equation}\label{eq:Cheeger2}
    \frac{1}{2}h(G^+)^2\leq 2-\lambda_{n-1}(G^+)\leq 2h(G^+).
\end{equation}Also, the Cheeger cut can be approximated based on the sign of a given eigenfunction of $\lambda_{n-1}(G^+)$. We shall use this equivalent formulation of the Cheeger inequalities in order to prove a generalization for uniform hypergraphs.\newline
In particular, given a connected, $k$-uniform hypergraph $\Gamma$, we will see it as an \textsl{oriented hypergraph} \cite{ReffRusnak} with only positive signs and we will consider the corresponding hypergraph normalized Laplacian $L(\Gamma)$ defined in \cite{Hypergraphs}. We will define a generalized Cheeger constant $h(\Gamma)$ for $\Gamma$ that coincides with the classical one in the particular case of graphs and we will prove, in Theorem \ref{thm:main} below, that
\begin{equation*}
    \frac{1}{2(k-1)} h(\Gamma)^2\leq k-\lambda_{n-1}(\Gamma)\leq 2(k-1)h(\Gamma).
\end{equation*}Clearly, the above inequalities generalize \eqref{eq:Cheeger2}, therefore \eqref{eq:Cheeger1}. Moreover, the proof will suggest that the eigenfunctions of $\lambda_{n-1}(\Gamma)$ can be used to approximate the Cheeger cut. \newline

\textbf{Related work.} It is worth mentioning some related work that is present in literature. In \cite{MulasZhang}, some Cheeger-like inequalities are shown for the smallest nonzero eigenvalue of $L(\Gamma)$, for restricted classes of hypergraphs which satisfy either only a generalized Cheeger upper bound or only a generalized Cheeger lower bound. In \cite{related2018,related2018_1,related3,related2020}, Cheeger-type inequalities are shown for other operators on hypergraphs.\newline

\textbf{Structure of the paper.} In Section \ref{Section:def} we give the preliminary definitions and in Section \ref{Section:main} we present the main results. In Section \ref{Section:Upper} we prove the upper Cheeger inequality and in Section \ref{Section:Lower} we prove the lower bound.

\section{Preliminary definitions}\label{Section:def}

\begin{definition}[\cite{ReffRusnak}]
	An \textsl{oriented hypergraph} is a triple $\Gamma=(V,E,\psi_\Gamma)$ such that $V$ is a finite set of vertices, $E$ is a finite multiset of elements $e\in \mathcal{P}(V)\setminus\{\emptyset\}$ called \textsl{edges}, while $\psi_\Gamma:(V,E)\rightarrow \{-1,0,+1\}$ is the \textsl{incidence function} and it is such that 
	\begin{equation*}
	    \psi_\Gamma(v,e)\neq 0 \iff v\in e.
	\end{equation*}
	Two vertices $i\neq j$ are \textsl{co-oriented in $e$} if $\psi_\Gamma(v,e)=\psi_\Gamma(w,e)\neq 0$ and they are \textsl{anti-oriented in $e$} if $\psi_\Gamma(v,e)=-\psi_\Gamma(w,e)\neq 0$.
\end{definition}

We fix, from here on, an oriented hypergraph $\Gamma=(V,E,\psi_\Gamma)$ on $n$ vertices $v_1,\ldots,v_n$.

\begin{definition}
The \textsl{degree} of a vertex $v$, denoted $\deg(v)$, is the number of edges containing $v$. The \textsl{cardinality} of a edge $e$, denoted $|e|$, is the number of vertices that are contained in $e$. $\Gamma$ is \textsl{$d$-regular} if $\deg(v)=d$ is constant for all $v\in V$; it is \textsl{$k$-uniform} if $|e|=k$ is constant for all $e\in E$. 
\end{definition}

\begin{remark}
Signed graphs can be seen as $2$-uniform oriented hypergraphs such that $E$ is a set. Simple graphs can be seen as signed graphs such that, for each $e\in E$, there exists a unique $v\in V$ with $\psi_\Gamma(v,e)=1$ and there exists a unique $w\in V$ with $\psi_\Gamma(w,e)=-1$. Classical hypergraphs (the ones we are going to consider) can be seen as oriented hypergraphs such that
	\begin{equation*}
	    \psi_\Gamma(v,e)= 1 \iff v\in e.
	\end{equation*}
\end{remark}

\begin{definition}
$\Gamma$ is \textsl{connected} if, for every pair of vertices $v,w\in V$, there exists a path that connects $v$ and $w$, i.e., there exist $w_1,\dots,w_m\in V$ and $e_1,\dots,e_{m-1}\in E$ such that:
\begin{itemize}
\item $w_1=v$;
\item $w_m=w$;
\item $\{w_i,w_{i+1}\}\subseteq e_i$ for each $i=1,\dots,m-1$.
\end{itemize}
\end{definition}

For simplicity, we shall assume that $\Gamma$ is connected and has no vertices of degree zero. These assumptions are not restrictive, since the spectrum of a hypergraph is given by the union of the spectra of its connected components \cite{Sharp}, while each vertex of degree zero simply produces $0$ as eigenvalue \cite{Chung}. We also assume that, for all $v\in V$,
\begin{equation}\label{eq:assumption}
    \deg(v)\leq \sum_{w\neq v}\deg(w).
\end{equation} This is always true in the case of graphs and we will need this assumption in the proof of the main theorem.

\begin{definition}[\cite{Hypergraphs}]
The \textsl{degree matrix} of $\Gamma$ is the $n\times n$ diagonal matrix
\begin{equation*}
    D=D(\Gamma):=\diag\bigl(\deg(v_1),\ldots,\deg(v_n)\bigr).
\end{equation*}
The \textsl{adjacency matrix} of $\Gamma$ is the $n\times n$ matrix $A=A(\Gamma):=(A_{ij})_{ij},$ where $A_{ii}:=0$ for each $i=1,\ldots,n$ and, for $i\neq j$,
\begin{align*}
        A_{ij}:=& \biggl|\{\text{edges in which }v_i \text{ and }v_j\text{ are anti-oriented}\}\biggr|+\\
        &-\biggl|\{\text{edges in which }v_i \text{ and }v_j\text{ are co-oriented}\}\biggr|.
\end{align*}The \textsl{normalized Laplacian} of $\Gamma$ is the $n\times n$ matrix 
\begin{equation*}
    L=L(\Gamma):=\id-D^{-1}A.
\end{equation*}
\end{definition}
\begin{remark}
If $\Gamma$ is a simple graph, the adjacency matrix has $(0,1)$-entries while, if $\Gamma$ is a classical hypergraph (seen as an oriented hypergraph such that the incidence function has values in $\{0,+1\}$), then the adjacency matrix  has nonpositive entries.
\end{remark}
From here on we shall assume that $\Gamma$ is a $k$-uniform, classical hypergraph, seen as an oriented hypergraph such that the incidence function has values in $\{0,+1\}$.\newline
As shown in \cite{Hypergraphs}, $L$ has $n$ real, nonnegative eigenvalues, counted with multiplicity. We denote them as
\begin{equation*}
		\lambda_1\leq\ldots\leq \lambda_n.
\end{equation*}Moreover, as shown in \cite{Sharp}, since $\Gamma$ is connected and $k$-uniform, $\lambda_n=k$ and the constant functions are the corresponding eigenfunctions. By the Courant-Fischer-Weyl min-max principle (cf.\ \cite{Hypergraphs}), the second largest eigenvalue of $L$ can be characterized in terms of the \textsl{Rayleigh quotient} of a nonzero function $f:V\rightarrow \mathbb{R}$,
\begin{equation*}
    \RQ(f):=\frac{\sum_{e\in E}\left(\sum_{v\in e}f(v)\right)^2}{\sum_{v\in V}\deg(v)f(v)^2}.
\end{equation*}In particular,
\begin{equation}\label{eq:n-1}
\lambda_{n-1}=\max_{f\perp \mathbf{1}}\RQ(f),
\end{equation}where the condition $f\perp \mathbf{1}$ denotes the \textsl{orthogonality to the constants},
\begin{equation*}
    \sum_{v\in V}\deg(v)f(v)=0,
\end{equation*}derived from the fact that the eigenfunctions corresponding to $\lambda_n$ are the constant functions (cf.\ \cite{Hypergraphs}).\newline
We now introduce the generalized Cheeger constant that will be used for bounding $k-\lambda_{n-1}$.

\begin{definition}Given $S\subseteq V$, we let $\bar{S}:=V\setminus S$, $\vol(S):=\sum_{v\in S}\deg(v)$ and
\begin{equation*}
    E_r(S):=\{e\in E: |e\cap S|=r\},
\end{equation*}for $r\in \{1,\ldots,k\}$.
\end{definition}
\begin{remark}
Clearly, for each $r\in \{1,\ldots,r\}$, $E_r(S)=E_{k-r}(\bar{S})$. Moreover,
\begin{equation*}
    E_k(S)=\{e\in E:e\subseteq S\},
\end{equation*}
\begin{equation*}
    E_0(S)=\{e\in E:e\subseteq \bar{S}\}
\end{equation*}and 
\begin{equation*}
    \vol(S)=\sum_{v\in S}\deg(v)=\sum_{r=1}^{k}r |E_r(S)|.
\end{equation*}
\end{remark}

\begin{definition}
Given $\emptyset\neq S\subsetneq V$,
\begin{equation*}
    h(S):=\frac{\sum_{r=1}^{k-1}|E_r(S)|r(k-r)}{\min\{\vol(S),\vol(\bar{S})\}}.
\end{equation*}The \textsl{Cheeger constant} of $\Gamma$ is
\begin{equation*}
    h:=\min_{\emptyset\neq S\subsetneq V}h(S).
\end{equation*}
\end{definition}
 \begin{remark}Observe that the quantity
 \begin{equation*}
     \sum_{r=1}^{k-1}|E_r(S)|r(k-r)
 \end{equation*}appearing in the numerator of $h(S)$ counts the number of pairwise connections between $S$ and $\bar{S}$. Furthermore, if $\Gamma$ is a graph, then $k=2$, $E_1(S)$ is the set of edges between $S$ and $\bar{S}$, and the Cheeger constant defined above coincides with the one introduced by Pólya and Szegő.
\end{remark}

\section{Main Results}\label{Section:main}
\subsection{Cheeger inequalities}
Our main result is the following theorem.
\begin{theorem}\label{thm:main}Let $\Gamma$ be a connected, $k$-uniform hypergraph. Then,
\begin{equation*}
    \frac{1}{2(k-1)} h^2\leq k-\lambda_{n-1}\leq 2(k-1)h.
\end{equation*}\end{theorem}
\begin{remark}
Theorem \ref{thm:main} generalizes \eqref{eq:Cheeger2} which is, on its turn, equivalent to the classical Cheeger inequalities in \eqref{eq:Cheeger1}.
\end{remark}
We split the proof of Theorem \ref{thm:main} into two parts: in Section \ref{Section:Upper} we prove the upper bound and in Section \ref{Section:Lower} we prove the lower bound. Our proofs are inspired by the graph case and in particular by the proof method in \cite[Lemma 2.1]{Chung} for the upper bound; by the proof method in \cite[Theorem 2.2]{Chung} for the lower bound. However, the proofs presented here for hypergraphs are much longer and more complicated than those for graphs. Both proofs make use of the fact that the eigenfunctions corresponding to $\lambda_{n-1}$ are orthogonal to the constants and, as we already observed, this is a consequence of the fact that $\Gamma$ is uniform.
\subsection{Cheeger cut}
The proof of Theorem \ref{thm:main} will also suggest that, as in the graph case, the Cheeger cut of $\Gamma$ can be approximated by the sets
\begin{equation*}
    \{v\in V:f(v)\geq 0\} \quad \text{and} \quad \{v\in V:f(v)<0\},
\end{equation*}for a given eigenfunction $f:V\rightarrow \mathbb{R}$ of $\lambda_{n-1}$. This gives a generalized method of spectral clustering for uniform hypergraphs.

\begin{example}\label{ex:hyp}
Let $\Gamma$ be the hypergraph in Figure \ref{fig:hyp}, with vertex set $V=\{v_1,\ldots,v_6\}$ and edge set $E=\{e_1,e_2,e_3\}$ such that:
\begin{itemize}
    \item $e_1=\{v_1,v_2,v_3\}$;
    \item $e_2=\{v_3,v_4,v_5\}$;
    \item $e_3=\{v_4,v_5,v_6\}$.
\end{itemize}Then, $ D=\textrm{diag}(1,1,2,2,2,1)$,
\begin{equation*}
   A=-
\begin{pmatrix}
  \begin{matrix}
 0 & 1 & 1 & 0& 0 & 0 \\
 1 & 0 & 1 & 0& 0 & 0 \\
 1 & 1 & 0 & 1& 1 & 0 \\
 0 & 0 & 1 & 0& 2 & 1 \\
 0 & 0 & 1 & 2 & 0 & 1 \\
 0 & 0 & 0 & 1& 1 & 0 
  \end{matrix}
\end{pmatrix}
\end{equation*}and
\begin{equation*}
  L=\id-D^{-1}A=
\begin{pmatrix}
   \begin{matrix}
 1 & 1 & 1 & 0& 0 & 0 \\
 1 & 1 & 1 & 0& 0 & 0 \\
0.5 &0.5 & 1 &0.5&0.5 & 0 \\
 0 & 0 &0.5 & 1& 1 &0.5 \\
 0 & 0 &0.5 & 1 & 1 &0.5 \\
 0 & 0 & 0 & 1& 1 & 1
  \end{matrix}
\end{pmatrix}.
\end{equation*}One can check that $\lambda_{n-1}=\frac{3+\sqrt{3}}{2}$ and a corresponding eigenfunction is
\begin{equation*}
  f= \biggl(-\frac{1+\sqrt{3}}{2},-\frac{1+\sqrt{3}}{2},-\frac{1}{2},\frac{1+\sqrt{3}}{4},\frac{1+\sqrt{3}}{4},1\biggr).
\end{equation*}Using $f$ for approximating the Cheeger cut gives
\begin{equation*}
    \{v_1,v_2,v_3\}\quad\text{and}\quad \{v_4,v_5,v_6\},
\end{equation*}as one would expect.

\begin{figure}
    \centering
    \includegraphics[width=9cm]{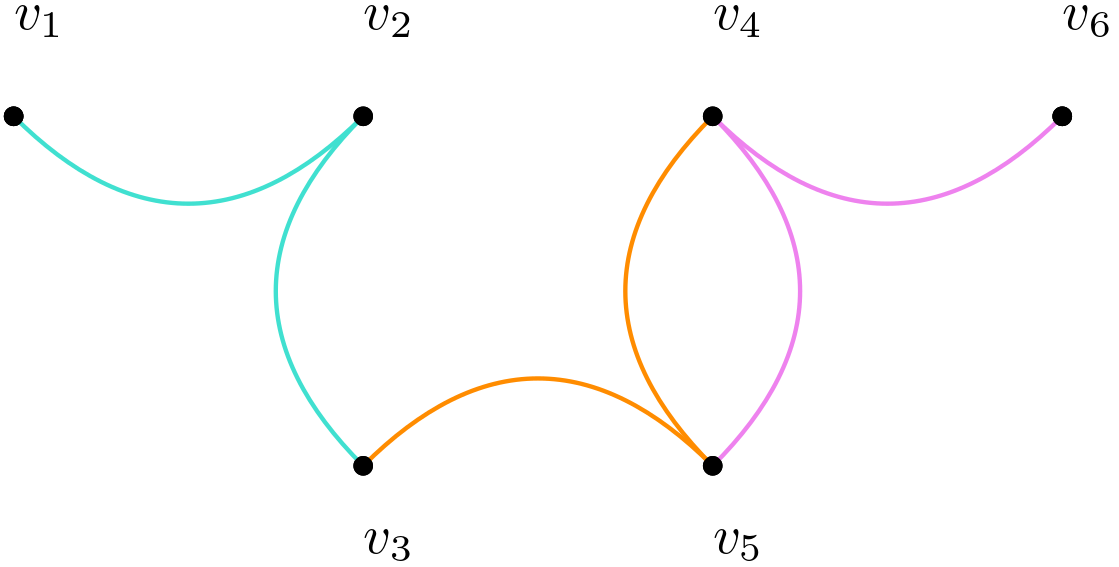}
    \caption{The hypergraph in Example \ref{ex:hyp}.}
    \label{fig:hyp}
\end{figure}
\end{example}

\subsection{Key idea} As we argued in the Introduction, the key idea used in this paper is to first reformulate the graph Cheeger inequalities in terms of the signless normalized Laplacian and then generalize them for the second largest eigenvalue of the hypergraph normalized Laplacian. The first step is fundamental. In \cite{MulasZhang}, for instance, an attempt to formulate generalized Cheeger inequalities in terms of the first nonzero eigenvalue of the hypergraph normalized Laplacian was made, but it led to generalizations for restricted classes of hypergraphs, either only for the Cheeger upper bound or only for the Cheeger lower bound. The reason is that the properties of the smallest eigenvalues of the graph Laplacian are preserved, in the general case, by the largest eigenvalues of the Laplacian.\newline
This change of point of view can allow us also to generalize, to the case of uniform hypergraphs, the fact that the multiplicity of $0$ for $L$ counts the number of connected components of a simple graph \cite{Chung}. In terms of the signless Laplacian, this is equivalent to saying that the multiplicity of $2$ for $L^+$ counts the number of connected components of a simple graph and, on its turn, this is equivalent to saying that the multiplicity of $2$ of $L$ equals the number of connected components in the case of a signed graph in which each edge has a positive sign. While this property cannot be generalized for hypergraphs in terms of the multiplicity of $0$ (cf.\ \cite{Hypergraphs}), it can be generalized in terms of the multiplicity of $\lambda_n$, as follows.
\begin{theorem}
If $\Gamma$ is a $k$-uniform hypergraph, then the multiplicity of $k$ equals the number of connected components of $\Gamma$.
\end{theorem}
\begin{proof}
It follows from the fact that, as shown in \cite{Sharp}, a connected $k$-uniform hypergraph has eigenvalue $\lambda_n=k$ and the corresponding eigenfunctions are exactly the constant functions.
\end{proof}
\subsection{Vertex cut for regular hypergraphs} Given a hypergraph $\Gamma=(V,E)$ on $n$ nodes $v_1,\ldots,v_n$ and $m$ edges $e_1,\ldots,e_m$, its \textsl{dual hypergraph} is $\Gamma^*:=(V^*,E^*)$, where:
\begin{itemize}
    \item $V^*:=\{v_1^*,\ldots,v_m^*\}$;
     \item $E^*:=\{e_1^*,\ldots,e_n^*\}$;
     \item $v_j^*\in e_j^*$ in $\Gamma^*$ if and only if $v_i\in e_j$ in $\Gamma$.
\end{itemize}Therefore, the vertices of $\Gamma$ correspond to the edges of $\Gamma^*$ and vice versa. In particular, if $\Gamma$ is $d$-regular, then $\Gamma^*$ is $d$-uniform. In this case, we can apply Theorem \ref{thm:main} to $\Gamma^*$ and the edge cut on $\Gamma^*$ can be translated into a \textsl{vertex cut} on $\Gamma$, as follows.

\begin{definition}Let $\Gamma=(V,E)$ be a $d$-regular hypergraph. Given $\emptyset \neq F\subsetneq E$, let $\bar{F}:=E\setminus F$, $\vol(F):=\sum_{e\in F}|e|$ and
\begin{equation*}
    V_r(F):=\{v\in V: v \text{ belongs to }r\text{ edges in }F\},
\end{equation*}for $r\in \{1,\ldots,d\}$. Let also
\begin{equation*}
    h_*(F):=\frac{\sum_{r=1}^{d-1}|V_r(F)|r(d-r)}{\min\{\vol(F),\vol(\bar{F})\}}.
\end{equation*}The \textsl{vertex Cheeger constant} of $\Gamma$ is
\begin{equation*}
    h_*:=\min_{\emptyset\neq F\subsetneq E}h_*(F).
\end{equation*}
\end{definition}
\begin{corollary}
Let $\Gamma$ be a connected, $d$-regular hypergraph. Then,
\begin{equation*}
    \frac{1}{2(d-1)} h_*^2\leq d-\lambda_{m-1}(\Gamma^*)\leq 2(d-1)h_*.
\end{equation*}
\end{corollary}
\begin{proof}
It follows from Theorem \ref{thm:main} applied to $\Gamma^*$.
\end{proof}In particular, using the signs of an eigenfunction of $\lambda_{m-1}(\Gamma^*)$, one can give an edge cut for $\Gamma^*$ corresponding to a vertex cut for $\Gamma$.

\subsection{Bipartite uniform hypergraphs}
For future directions, it will be interesting to see whether the results presented here could be extended to classical hypergraphs that are not necessarily uniform and, more generally, to oriented hypergraphs. We can already say something for \textsl{bipartite hypergraphs:} oriented hypergraphs whose vertex set can be partitioned into two disjoint subsets as $V=V_1\sqcup V_2$, such that each edge contain all its positive incidences in $V_1$ and all its negative incidences in $V_2$, or vice versa. Bipartite hypergraphs generalize bipartite graphs and, as shown in \cite{AndreottiMulas}, a bipartite hypergraph $\Gamma=(V,E,\psi_\Gamma)$ has the same spectrum as $\Gamma^+:=(V,E,\psi_{\Gamma^+})$, where $\psi^+$ is such that
	\begin{equation*}
	    \psi_{\Gamma^+}(v,e)= 1 \iff v\in e.
	\end{equation*}This implies that the Cheeger inequalities in Theorem \ref{thm:main} also hold for bipartite $k$-uniform hypergraphs. However, since the eigenfunctions of $\Gamma$ and $\Gamma^+$ differ by changes of signs, in this case we cannot approximate the Cheeger cut by 
	\begin{equation*}
    \{v\in V:f(v)\geq 0\} \quad \text{and} \quad \{v\in V:f(v)<0\},
\end{equation*}for a given eigenfunction $f:V\rightarrow \mathbb{R}$ of $\lambda_{n-1}$.

\section{Proof of the upper bound}\label{Section:Upper}
\begin{theorem}Let $\Gamma$ be a connected, $k$-uniform hypergraph. Then,
\begin{equation*}
   k-\lambda_{n-1}\leq 2(k-1)h.
\end{equation*}\end{theorem}
\begin{proof}
Let $\emptyset\neq S\subsetneq V$ be such that $h=h(S)=h(\bar{S})$, and assume, without loss of generality, that $\vol(S)\leq \vol(\bar{S})$. Let
\begin{equation*}
    \alpha:=\frac{\vol (S)}{\vol (\bar{S})}\leq 1
\end{equation*}and let $f$ be a function on $V$ defined by
\begin{equation*}
    f(v):=\begin{cases}
    1 &\text{ if }v\in S\\
-\alpha&\text{ if }v\in \bar{S}.
    \end{cases}
\end{equation*}By construction of $f$, $\sum_{v\in V}\deg(v) f(v)=0$, that is, $f$ is orthogonal to the constants. Thus, by \eqref{eq:n-1},
\begin{align*}
    \lambda_{n-1}&\geq \RQ(f)\\
    &=\frac{\sum_{e\in E}\bigl(\sum_{v\in e}f(v)\bigr)^2}{\sum_{v\in V}\deg(v)f(v)^2}\\
    &=\frac{\sum_{e\in E}\bigl(\sum_{v\in e\cap S}1-\sum_{v\in e\cap \bar{S}}\alpha\bigr)^2}{\vol(S)+\alpha^2\vol(\bar{S})}\\
 \bigl(\text{by }\alpha^2\vol(\bar{S})=\alpha \vol(S) \bigr)  &=\frac{\sum_{e\in E}\bigl(|e\cap S|-\alpha|e\cap \bar{S}|\bigr)^2}{(\alpha+1)\vol(S)}\\
   \bigl(\text{by }|e|=k\,\,\forall e\in E \bigr) &=\frac{\sum_{e\in E}\bigl(k-|e\cap \bar{S}| -\alpha |e\cap \bar{S}|\bigr)^2}{(\alpha+1)\vol(S)}\\
     &=\frac{\sum_{e\in E}\bigl(k-|e\cap \bar{S}|\cdot (\alpha+1) \bigr)^2}{(\alpha+1)\vol(S)}\\
  &=\frac{|E| k^2}{(\alpha+1)\vol(S)}+\frac{\sum_{e\in E}|e\cap \bar{S}|^2(\alpha+1)}{\vol(S)}-2k\cdot \frac{\sum_{e\in E}|e\cap \bar{S}| }{\vol(S)}\\
 \bigl(\text{by }\sum_{e\in E}|e\cap \bar{S}|=\vol(\bar{S}) \bigr)\quad &=\frac{|E| k^2}{(\alpha+1)\vol(S)}+\frac{\sum_{e\in E}|e\cap \bar{S}|^2(\alpha+1)}{\vol(S)}-\frac{2k}{\alpha}\\
&=\frac{|E| k^2}{(\alpha+1)\vol(S)}+\frac{\sum_{r=1}^k\sum_{e\in E:|e\cap \bar{S}|=r}r^2(\alpha+1)}{\vol(S)}-\frac{2k}{\alpha}\\
\bigl(\text{by }|E|k=\vol(V) \bigr)\quad &=\frac{k\cdot \vol(V)}{(\alpha+1)\vol(S)}+ \frac{(\alpha+1)\cdot\sum_{r=1}^k|E_r(\bar{S})|r^2}{\vol(S)}-\frac{2k}{\alpha}.
\end{align*}
Now, observe that
\begin{equation*}
    \frac{k\cdot \vol(V)}{(\alpha+1)\vol(S)}=\frac{k\cdot( \vol(S)+\vol(\bar{S}))}{(\alpha+1)\vol(S)}=\frac{k}{\alpha+1}+\frac{k}{\alpha(\alpha+1)}
\end{equation*}and we have that
\begin{equation*}
    \frac{k}{\alpha+1}+\frac{k}{\alpha(\alpha+1)}-\frac{2k}{\alpha}=\frac{k(\alpha+1-2\alpha-2)}{\alpha(\alpha+1)}=-\frac{k}{\alpha}.
\end{equation*}Therefore, by putting everything together,
\begin{align*}
    \lambda_{n-1}&\geq \RQ(f)\\
    &=\frac{k\cdot \vol(V)}{(\alpha+1)\vol(S)}+ \frac{(\alpha+1)\cdot\sum_{r=1}^k|E_r(\bar{S})|r^2}{\vol(S)}-\frac{2k}{\alpha}\\
    &=\frac{(\alpha+1)\cdot\sum_{r=1}^k|E_r(\bar{S})|r^2}{\vol(S)}-\frac{k}{\alpha}\\
    &= \frac{(\alpha+1)\cdot\sum_{r=1}^{k-1}|E_r(\bar{S})|r^2}{\vol(S)}+\frac{(\alpha+1)\cdot|E_k(\bar{S})|k^2}{\vol(S)}-\frac{k}{\alpha}\\
    &\geq \frac{(\alpha+1)\cdot\sum_{r=1}^{k-1}|E_r(\bar{S})|r}{\vol(S)}+\frac{(\alpha+1)\cdot |E_k(\bar{S})|k}{\vol(S)}+\frac{(\alpha+1)\cdot|E_k(\bar{S})|k(k-1)}{\vol(S)}-\frac{k}{\alpha}\\
    &=\frac{(\alpha+1)\cdot\sum_{r=1}^{k}|E_r(\bar{S})|r}{\vol(S)}+\frac{(\alpha+1)\cdot|E_k(\bar{S})|k(k-1)}{\vol(S)}-\frac{k}{\alpha}.
\end{align*}Now, since $\vol(\bar{S})=\sum_{r=1}^k|E_r(\bar{S})|r$,
\begin{equation*}
    |E_k(\bar{S})|k=\vol(\bar{S})-\sum_{r=1}^{k-1}|E_r(\bar{S})|r.
\end{equation*}Hence,
\begin{align*}
    \lambda_{n-1}&\geq 
    \frac{(\alpha+1)\vol(\bar{S})}{\vol(S)}+(\alpha+1)\cdot (k-1)\cdot\Biggl(\frac{|E_k(\bar{S})|k}{\vol(S)}\Biggr)-\frac{k}{\alpha}\\
    &= \frac{(\alpha+1)\vol(\bar{S})}{\vol(S)}+(\alpha+1)\cdot (k-1)\cdot\Biggl(\frac{\vol(\bar{S})}{\vol(S)}-\frac{\sum_{r=1}^{k-1}|E_r(\bar{S})|r}{\vol(S)}\Biggr)-\frac{k}{\alpha}\\
    &=\frac{(\alpha+1)}{\alpha}-\frac{k}{\alpha}+(\alpha+1)\cdot (k-1)\cdot\Biggl(\frac{1}{\alpha}-\frac{\sum_{r=1}^{k-1}|E_r(\bar{S})|r}{\vol(S)}\Biggr)\\
     &=\frac{\alpha+1-k+(\alpha+1)(k-1)}{\alpha}-(\alpha+1)\cdot (k-1)\cdot\Biggl(\frac{\sum_{r=1}^{k-1}|E_r(\bar{S})|r}{\vol(S)}\Biggr)\\
    &=k-(\alpha+1)\cdot (k-1)\cdot\Biggl(\frac{\sum_{r=1}^{k-1}|E_r(\bar{S})|r}{\vol(S)}\Biggr)\\
   \bigl(\text{by }\alpha+1\leq 2\bigr)\quad &\geq k-2(k-1)\Biggl(\frac{\sum_{r=1}^{k-1}|E_r(\bar{S})|r}{\vol(S)}\Biggr)\\
   &\geq k-2(k-1)\Biggl(\frac{\sum_{r=1}^{k-1}|E_r(\bar{S})|r(k-r)}{\vol(S)}\Biggr)\\
   &=k-2(k-1)h.
\end{align*}The claim follows.
\end{proof}

\section{Proof of the lower bound}\label{Section:Lower}
\begin{theorem}Let $\Gamma$ be a connected, $k$-uniform hypergraph. Then,
\begin{equation*}
 k-\lambda_{n-1}\geq \frac{1}{2(k-1)} h^2.
\end{equation*}\end{theorem}

\begin{proof}
We follow and generalize the proof method of \cite[Theorem 2.2]{Chung}.\newline
Let $f$ be an eigenfunction for $L$ with eigenvalue $\lambda_{n-1}$. Without loss of generality, we relabel the vertices so that
\begin{equation*}
    f(v_i)\geq f(v_{i+1}),\,\text{ for }i=1,\ldots,n-1.
\end{equation*}Let $S_i:=\{v_1,\ldots,v_i\}$ and let 
\begin{equation*}
    t:=\max\{i:\vol(S_i)\leq \vol(\bar{S_i})\}.
\end{equation*}Since we are assuming \eqref{eq:assumption}, $t$ is well defined. Now, since $f$ is orthogonal to the constants, $\sum_{v\in V}f(v)\deg(v)=0$. Hence, 
\begin{align*}
    \sum_{v\in V}\deg(v)\biggl(f(v)+f(v_t)\biggr)^2&=\sum_{v\in V}\deg(v) f(v)^2+f(v_t)^2\vol(V)\\
    & \geq \sum_{v\in V}\deg(v) f(v)^2.
\end{align*}This implies that
\begin{align*}
    k-\lambda_{n-1}&=k-\frac{\sum_{e\in E}\bigl(\sum_{v\in e}f(v)\bigr)^2}{\sum_{v\in V}\deg(v)f(v)^2}\\
    &=\frac{k\cdot \sum_{v\in V}\deg(v)f(v)^2-\sum_{e\in E}\bigl(\sum_{v\in e}f(v)\bigr)^2}{\sum_{v\in V}\deg(v)f(v)^2}\\
    &\geq \frac{k\cdot \sum_{v\in V}\deg(v)f(v)^2-\sum_{e\in E}\bigl(\sum_{v\in e}f(v)\bigr)^2}{\sum_{v\in V}\deg(v)\biggl(f(v)+f(v_t)\biggr)^2}.
\end{align*}Now, for $v\in V$, let
\begin{equation*}
    f_+(v):=\begin{cases}f(v)+f(v_t) &\text{if }f(v)+f(v_t)\geq 0\\
    0 & \text{otherwise}
    \end{cases}
\end{equation*}and let
\begin{equation*}
    f_-(v):=\begin{cases}|f(v)+f(v_t)| &\text{if }f(v)+f(v_t)\leq 0\\
    0 & \text{otherwise.}
    \end{cases}
\end{equation*}Then, 
\begin{equation}\label{eq:+-1}
    f_+(v)+f_-(v)=|f(v)+f(v_t)| \quad \forall v\in V;
\end{equation}similarly
\begin{equation}\label{eq:+-2}
    f_+(v)^2+f_-(v)^2=\biggl(f(v)+f(v_t)\biggr)^2 \quad \forall v\in V
\end{equation}
and, by \eqref{eq:+-1},
\begin{equation*}
    \sum_{v\in V}\deg(v)\biggl(f(v)+f(v_t)\biggr)^2=\sum_{v\in V}\deg(v)\biggl(f_+(v)^2+f_-(v)^2\biggr).
\end{equation*}Moreover,
\begin{equation}\label{eq:todo}
  \sum_{e\in E}\Biggl(\sum_{v\in e}f(v)\Biggr)^2\leq \sum_{e\in E}\Biggl(\biggl(\sum_{v\in e}f_+(v)\biggr)^2+\biggl(\sum_{v\in e}f_-(v)\biggr)^2\Biggr)-|E|k^2f(v_t)^2.
\end{equation}
To see this, observe first that, for each $e\in E$,
\begin{align*}
&\biggl(\sum_{v\in e}f_+(v)\biggr)^2+\biggl(\sum_{v\in e}f_-(v)\biggr)^2\\
&=\sum_{v\in e}\biggl(f_+(v)^2+f_-(v)^2\biggr)+2\sum_{v\neq w:\{v,w\}\subseteq e}\biggl(f_+(v)f_+(w)+f_-(v)f_-(w)\biggr)\\
\bigl(\text{by \eqref{eq:+-2}}\bigr)\quad &=\sum_{v\in e}\biggl(f(v)+f(v_t)\biggr)^2+2\sum_{v\neq w:\{v,w\}\subseteq e}\biggl(f_+(v)f_+(w)+f_-(v)f_-(w)\biggr)\\
\bigl(\text{by construction of $f$}\bigr)\quad &\geq \sum_{v\in e}\biggl(f(v)+f(v_t)\biggr)^2+2\sum_{v\neq w:\{v,w\}\subseteq e}\biggl(f(v)+f(v_t)\biggr)\biggl(f(w)+f(v_t)\biggr)\\
&=\sum_{v\in e}f(v)^2+k\cdot f(v_t)^2+2f(v_t)\cdot \sum_{v\in e}f(v)+\\
&\quad +2\sum_{v\neq w:\{v,w\}\subseteq e}\Biggl(f(v)f(w)+f(v_t)f(v)+f(v_t)f(w)+f(v_t)^2\Biggr)\\
\bigl(\text{since }|e|=k\,\,\forall e\in E\bigr)\quad&=\sum_{v\in e}f(v)^2+k\cdot f(v_t)^2+2f(v_t)\cdot \sum_{v\in e}f(v)+\\
&\quad+2\sum_{v\neq w:\{v,w\}\subseteq e}f(v)f(w)+2f(v_t)(k-1)\sum_{v\in e}f(v)+k(k-1)f(v_t)^2\\
&=\sum_{v\in e}f(v)^2+2\sum_{v\neq w:\{v,w\}\subseteq e}f(v)f(w)+2f(v_t)k\cdot \sum_{v\in e}f(v)+k^2f(v_t)^2\\
&=\Biggl(\sum_{v\in e}f(v)\Biggr)^2+2f(v_t)k\cdot \sum_{v\in e}f(v)+k^2f(v_t)^2.
\end{align*}
In going from the third to the fourth line, we used the fact that, by definition of $f$,
\begin{equation}\label{eq:f+-}
    \biggl(f_+(v)f_+(w)+f_-(v)f_-(w)\biggr)\geq \biggl(f(v)+f(v_t)\biggr)\biggl(f(w)+f(v_t)\biggr), \text{ for }v\neq w.
\end{equation}
This can be seen in more detail by considering the following three cases.
\begin{itemize}
    \item Case 1:
    \begin{equation*}
      f(v)+f(v_t)\geq 0 \text{ and }f(w)+f(v_t)\geq 0.
    \end{equation*}In this case, by definition of $f$,
    \begin{equation*}
     f_+(v)f_+(w)+f_-(v)f_-(w)=  f_+(v)f_+(w)= \biggl(f(v)+f(v_t)\biggr)\biggl(f(w)+f(v_t)\biggr).
    \end{equation*}
    \item  Case 2:
    \begin{equation*}
      f(v)+f(v_t)\leq 0 \text{ and }f(w)+f(v_t)\leq 0.
    \end{equation*}In this case, by definition of $f$,
    \begin{equation*}
     f_+(v)f_+(w)+f_-(v)f_-(w)=  f_-(v)f_-(w)= \biggl|f(v)+f(v_t)\biggr|\cdot \biggl|f(w)+f(v_t)\biggr|.
    \end{equation*}
    \item Case 3:
    \begin{equation*}
      f(v)+f(v_t)\geq 0 \text{ and }f(w)+f(v_t)\leq 0, \text{ or vice versa}.
    \end{equation*}In this case, by definition of $f$,
    \begin{equation*}
     f_+(v)f_+(w)+f_-(v)f_-(w)=0,
    \end{equation*}while
    \begin{equation*}
        \biggl(f(v)+f(v_t)\biggr)\biggl(f(w)+f(v_t)\biggr)\leq 0.
    \end{equation*}
\end{itemize}
This proves \eqref{eq:f+-}. Therefore,
\begin{align*}
    &\sum_{e\in E}\Biggl(\biggl(\sum_{v\in e}f_+(v)\biggr)^2+\biggl(\sum_{v\in e}f_-(v)\biggr)^2\Biggr)\\
    &\geq \sum_{e\in E}\Biggl(\Biggl(\sum_{v\in e}f(v)\Biggr)^2+2f(v_t)k\cdot \sum_{v\in e}f(v)+k^2f(v_t)^2\Biggr)\\
    &=\sum_{e\in E}\Biggl(\sum_{v\in e}f(v)\Biggr)^2+2f(v_t)k\cdot \sum_{v\in V}\deg(v)f(v)+|E|k^2f(v_t)^2\\
    \bigl(\text{by $\sum_{v\in V}\deg(v)f(v)=0$}\bigr)\quad&=\sum_{e\in E}\Biggl(\sum_{v\in e}f(v)\Biggr)^2+|E|k^2f(v_t)^2.
\end{align*}Hence,
\begin{equation*}
    \sum_{e\in E}\Biggl(\sum_{v\in e}f(v)\Biggr)^2\leq \sum_{e\in E}\Biggl(\biggl(\sum_{v\in e}f_+(v)\biggr)^2+\biggl(\sum_{v\in e}f_-(v)\biggr)^2\Biggr)-|E|k^2f(v_t)^2.
\end{equation*}
This proves \eqref{eq:todo}. By putting everything together,
\begin{align*}
    k-\lambda_{n-1}&\geq \frac{k\cdot \sum_{v\in V}\deg(v)f(v)^2-\sum_{e\in E}\bigl(\sum_{v\in e}f(v)\bigr)^2}{\sum_{v\in V}\deg(v)\biggl(f(v)+f(v_t)\biggr)^2}\\
    &\geq \frac{k\cdot \sum_{v\in V}\deg(v)f(v)^2+|E|k^2f(v_t)^2-\sum_{e\in E}\Biggl(\bigl(\sum_{v\in e}f_+(v)\bigr)^2+\bigl(\sum_{v\in e}f_-(v)\bigr)^2\Biggr)}{\sum_{v\in V}\deg(v)\biggl(f_+(v)^2+f_-(v)^2\biggr)}\\
    &=\frac{k\cdot \sum_{v\in V}\deg(v)f(v)^2+|E|k^2f(v_t)^2}{\sum_{v\in V}\deg(v)\biggl(f(v)+f(v_t)\biggr)^2}-\frac{\sum_{e\in E}\Biggl(\bigl(\sum_{v\in e}f_+(v)\bigr)^2+\bigl(\sum_{v\in e}f_-(v)\bigr)^2\Biggr)}{\sum_{v\in V}\deg(v)\biggl(f_+(v)^2+f_-(v)^2\biggr)}\\
    &\geq \frac{k\cdot \sum_{v\in V}\deg(v)f(v)^2+|E|k^2f(v_t)^2}{\sum_{v\in V}\deg(v)\biggl(f(v)+f(v_t)\biggr)^2}-\max\{\RQ(f_+),\RQ(f_-)\},
\end{align*}since 
\begin{equation*}
    \frac{a+b}{c+d}\leq \max\{\frac{a}{c},\frac{b}{d}\}.
\end{equation*}Now assume, without loss of generality, that $\RQ(f_+)\geq \RQ(f_-)$. Then,
\begin{align*}
    k-\lambda_{n-1}&\geq \frac{k\cdot \sum_{v\in V}\deg(v)f(v)^2+|E|k^2f(v_t)^2}{\sum_{v\in V}\deg(v)\biggl(f(v)+f(v_t)\biggr)^2}-\RQ(f_+).
\end{align*}Now, by the orthogonality to the constants and since $\vol(V)=\sum_{v\in V}\deg(v)=|E|k$,
\begin{align*}
    \frac{k\cdot \sum_{v\in V}\deg(v)f(v)^2+|E|k^2f(v_t)^2}{\sum_{v\in V}\deg(v)\biggl(f(v)+f(v_t)\biggr)^2}&=\frac{k\cdot \sum_{v\in V}\deg(v)f(v)^2+|E|k^2f(v_t)^2}{\sum_{v\in V}\deg(v)\biggl(f(v)^2+f(v_t)^2\biggr)}\\
    &=k\cdot \frac{\sum_{v\in V}\deg(v)f(v)^2+|E|kf(v_t)^2}{\sum_{v\in V}\deg(v)f(v)^2+|E|kf(v_t)^2}\\
    &=k.
\end{align*}Hence, by letting $E(v,w)$ denote the set of edges that contain both $v$ and $w$,
\begin{align*}
k-\lambda_{n-1}&\geq k-\RQ(f_+)\\
    &=k-\frac{\sum_{e\in E}\biggl(\sum_{v\in e}f_+(v)\biggr)^2}{\sum_{v\in V}\deg(v)f_+(v)^2}\\
   &=\frac{k\cdot\biggl(\sum_{v\in V}\deg(v)f_+(v)^2\biggr)-\sum_{e\in E}\biggl(\sum_{v\in e}f_+(v)^2+2\sum_{\{v,w\}\subseteq e:v\neq w}f_+(v)f_+(w)\biggr)}{\sum_{v\in V}\deg(v)f_+(v)^2}\\
   &=\frac{(k-1)\cdot\biggl(\sum_{v\in V}\deg(v)f_+(v)^2\biggr)-2\sum_{v\neq w}|E(v,w)|f_+(v)f_+(w)}{\sum_{v\in V}\deg(v)f_+(v)^2}\\
   &=\frac{\sum_{e\in E}\sum_{\{v,w\}\subseteq e}\biggl(f_+(v)-f_+(w)\biggr)^2}{\sum_{v\in V}\deg(v)f_+(v)^2}\\
   &=\frac{\sum_{e\in E}\sum_{\{v,w\}\subseteq e}\biggl(f_+(v)-f_+(w)\biggr)^2}{\sum_{v\in V}\deg(v)f_+(v)^2}\cdot\frac{\sum_{e\in E}\sum_{\{v,w\}\subseteq e}\biggl(f_+(v)+f_+(w)\biggr)^2}{\sum_{e\in E}\sum_{\{v,w\}\subseteq e}\biggl(f_+(v)+f_+(w)\biggr)^2}\\
   &\geq \frac{\Biggl(\sum_{v\neq w}|E(v,w)|\cdot\biggl(f_+(v)-f_+(w)\biggr)^2\Biggr)\cdot \Biggl(\sum_{v\neq w}|E(v,w)|\cdot \biggl(f_+(v)+f_+(w)\biggr)^2\Biggr)}{2(k-1)\biggl(\sum_{v\in V}\deg(v)f_+(v)^2\biggr)^2},
    \end{align*}using the inequality $ \bigl(f_+(v)+f_+(w)\bigr)^2 \leq 2 \bigl(f_+(v)^2+f_+(w)^2\bigr)$ in the denominator.\newline Now, by the Cauchy-Schwarz inequality, the numerator in the last line is such that
    \begin{align*}
&\Biggl(\sum_{v\neq w}|E(v,w)|\cdot\biggl(f_+(v)-f_+(w)\biggr)^2\Biggr)\cdot \Biggl(\sum_{v\neq w}|E(v,w)|\cdot \biggl(f_+(v)+f_+(w)\biggr)^2\Biggr) \\
&\geq \Biggl(\sum_{v\neq w}|E(v,w)|\cdot\biggl(f_+(v)-f_+(w)\biggr)\biggl(f_+(v)+f_+(w)\biggr)\Biggr)^2\\
&=\Biggl(\sum_{v\neq w}|E(v,w)|\cdot\biggl(f_+(v)^2-f_+(w)^2\biggr)\Biggr)^2.
\end{align*}Hence,
\begin{equation*}
    k-\lambda_{n-1}\geq \frac{\Biggl(\sum_{v\neq w}|E(v,w)|\cdot\biggl(f_+(v)^2-f_+(w)^2\biggr)\Biggr)^2}{2(k-1)\biggl(\sum_{v\in V}\deg(v)f_+(v)^2\biggr)^2}.
\end{equation*}Now,
\begin{align*}
   &\sum_{v\neq w}|E(v,w)|\cdot\biggl(f_+(v)^2-f_+(w)^2\biggr)\\&=\sum_{a<c}|E(v_a,v_c)|\cdot\biggl(f_+(v_a)^2-f_+(v_c)^2\biggr)\\
   &=\sum_{a<c}|E(v_a,v_c)|\cdot\Biggl(\sum_{i=a}^{c-1}f_+(v_i)^2-f_+(v_{i+1})^2\Biggr)\\
   &=\sum_{a<c}\sum_{i=a}^{c-1}|E(v_a,v_c)|\cdot\biggl(f_+(v_i)^2-f_+(v_{i+1})^2\biggr)\\
   &=\sum_{i=1}^{n-1}\sum_{a\leq i}\sum_{c>i}|E(v_a,v_c)|\cdot\biggl(f_+(v_i)^2-f_+(v_{i+1})^2\biggr)\\
   &=\sum_{i=1}^{n-1}\sum_{v_a\in S_i}\sum_{v_c\in \bar{S_i}}|E(v_a,v_c)|\cdot\biggl(f_+(v_i)^2-f_+(v_{i+1})^2\biggr)\\
   &=\sum_{i=1}^{n-1} \sum_{r=1}^{k-1}r(k-r)|E_r(S_i)|\cdot\biggl(f_+(v_i)^2-f_+(v_{i+1})^2\biggr).
\end{align*}It follows that
\begin{equation*}
    k-\lambda_{n-1}\geq
\frac{\Biggl(\sum_{i=1}^{n-1}\sum_{r=1}^{k-1}r(k-r)|E_r(S_i)|\cdot\biggl(f_+(v_i)^2-f_+(v_{i+1})^2\biggr)\Biggr)^2}{2(k-1)\biggl(\sum_{v\in V}\deg(v)f_+(v)^2\biggr)^2}.
\end{equation*}Now, for each $i=1,\ldots,n-1$, we let \begin{equation*}
    |\delta(S_i)|:=\sum_{r=1}^{k-1}r(k-r)|E_r(S_i)| \quad\text{and}\quad \widetilde{\vol}(S_i):=\min\{\vol(S_i),\vol(\bar{S_i})\},
\end{equation*}so that
\begin{equation*}
    h(S_i)=\frac{|\delta(S_i)|}{\widetilde{\vol}(S_i)}\geq h.
\end{equation*}Then,
\begin{align*}
   & \Biggl(\sum_{i=1}^{n-1}\sum_{r=1}^{k-1}r(k-r)|E_r(S_i)|\cdot\biggl(f_+(v_i)^2-f_+(v_{i+1})^2\biggr)\Biggr)^2\\
   &=\Biggl(\sum_{i=1}^{n-1}|\delta(S_i)|\cdot\biggl(f_+(v_i)^2-f_+(v_{i+1})^2\biggr)\Biggr)^2\\
   &\geq \Biggl(\sum_{i=1}^{n-1} h\cdot \widetilde{\vol}(S_i)\cdot\biggl(f_+(v_i)^2-f_+(v_{i+1})^2\biggr)\Biggr)^2\\
   &=h^2\cdot \Biggl(\widetilde{\vol}(S_1)f_+(v_1)^2+\sum_{i=2}^n \biggl(\widetilde{\vol}(S_i)-\widetilde{\vol}(S_{i-1})\biggr)f_+(v_i)^2 \Biggr)^2\\
   &=h^2\cdot \Biggl(\sum_{i=1}^n \deg(v_i) f_+(v_i)^2 \Biggr)^2,
\end{align*}where in the last line we have used the assumption \eqref{eq:assumption}. Putting everything together,
\begin{align*}
    k-\lambda_{n-1}&\geq \frac{\Biggl(\sum_i\sum_{r=1}^{k-1}r(k-r)|E_r(S_i)|\cdot\biggl(f_+(v_i)^2-f_+(v_{i+1})^2\biggr)\Biggr)^2}{2(k-1)\biggl(\sum_{v\in V}\deg(v)f_+(v)^2\biggr)^2}\\
    &\geq \frac{h^2}{2(k-1)}\cdot \frac{ \Biggl(\sum_{i=1}^n \deg(v_i)f_+(v_i)^2\Biggr)^2}{\biggl(\sum_{v\in V}\deg(v)f_+(v)^2\biggr)^2}\\
    &=\frac{h^2}{2(k-1)}.
\end{align*}
\end{proof}

\textbf{Acknowledgments.} The author would like to thank the anonymous referees for valuable suggestions, and Jürgen Jost (Max Planck Institute for Mathematics in the Sciences) for every time he told her not to be pessimistic about a possible generalization of the Cheeger inequalities to the case of hypergraphs.

\bibliographystyle{unsrt}
\bibliography{Cheeger25.5}

\begin{thebibliography}{10}

\bibitem{Polya}
G.~Pólya and S.~Szegő.
\newblock {Isoperimetric inequalities in mathematical physics}.
\newblock {\em Annals of Math. Studies}, 27, 1951.

\bibitem{CheegerPhD}
J.~Cheeger.
\newblock {A lower bound for the smallest eigenvalue of the Laplacian}.
\newblock In {\em Problems in Analysis}, pages 195--199. Princeton University
  Press, 1970.

\bibitem{Buser1}
P.~Buser.
\newblock {On Cheeger's Inequality $\lambda_1\geq h^2/4$}.
\newblock In {\em Geometry of the Laplace Operator (Proceedings of Symposia in
  Pure Mathematics)}, volume~36, pages 29--77, 1980.

\bibitem{Buser2}
P.~Buser.
\newblock A note on the isoperimetric constant.
\newblock {\em Annales scientifiques de l'\'Ecole Normale Sup\'erieure}, Ser.
  4, 15(2):213--230, 1982.

\bibitem{dodziuk}
J.~Dodziuk.
\newblock Difference equations, isoperimetric inequality and transience of
  certain random walks.
\newblock {\em Transactions of the American Mathematical Society},
  284(2):787--794, 1984.

\bibitem{Alon}
N.~Alon and V.~Milman.
\newblock {$\lambda_1$, isoperimetric inequalities for graphs, and
  superconcentrators}.
\newblock {\em J. Combin. Theory Ser. B}, 38:73--88, 1985.

\bibitem{Chung}
F.~Chung.
\newblock Spectral graph theory.
\newblock {\em American Mathematical Society}, 1997.

\bibitem{np}
A.~Szlam and X.~Bresson.
\newblock {Total Variation and Cheeger Cuts}.
\newblock In {\em Proceedings of the 27th International Conference on
  International Conference on Machine Learning}, ICML'10, pages 1039--1046,
  Madison, WI, USA, 2010. Omnipress.

\bibitem{appl2007}
D.A. Spielman and S.-H. Teng.
\newblock {Spectral partitioning works: Planar graphs and finite element
  meshes}.
\newblock {\em Linear Algebra and its Applications}, 421:284--305, 2007.

\bibitem{appl2008}
U.~von Luxburg, M.~Belkin, and O.~Bousquet.
\newblock Consistency of spectral clustering.
\newblock {\em Ann. Statist.}, 36(2):555--586, 04 2008.

\bibitem{appl2012}
A.-C. Castro, B.~Pelletier, and P.~Pudlo.
\newblock The normalized graph cut and cheeger constant: from discrete to
  continuous.
\newblock {\em Adv. in Appl. Probab.}, 44(4):907--93, 2012.

\bibitem{appl2015}
K.C. Chang, S.~Shao, and D.~Zhang.
\newblock The 1-laplacian cheeger cut: Theory and algorithms.
\newblock {\em Journal of Computational Mathematics}, 33(5):443--467, 2015.

\bibitem{Clustering}
U.~von Luxburg.
\newblock A tutorial on spectral clustering.
\newblock {\em Stat. Comput.}, 17(4):395--416, 2007.

\bibitem{KlamtHausTheis}
S.~Klamt, U.U. Haus, and F.~Theis.
\newblock Hypergraphs and cellular networks.
\newblock {\em PLoS Comp. Biol.}, 5(5):e1000385, 2009.

\bibitem{ZhangLiu}
Z.K. Zhang and C.~Liu.
\newblock A hypergraph model of social tagging networks.
\newblock {\em J. Stat. Mech.}, 2010(10):P10005, 2010.

\bibitem{neuro1}
R.~Mulas and N.M. Tran.
\newblock Minimal embedding dimensions of connected neural codes.
\newblock {\em Algebraic Statistics}, 11(1):99--106, 2020.

\bibitem{LanchierNeufer}
N.~Lanchier and J.~Neufer.
\newblock Stochastic dynamics on hypergraphs and the spatial majority rule
  model.
\newblock {\em J. Stat. Phys.}, 151(1):21--45, 2013.

\bibitem{BodoKatonaSimon}
{\'A}.~Bod{\'o}, G.Y. Katona, and P.L. Simon.
\newblock {SIS} epidemic propagation on hypergraphs.
\newblock {\em Bull. Math. Biol.}, 78(4):713--735, 2016.

\bibitem{ReffRusnak}
N.~Reff and L.~Rusnak.
\newblock {An oriented hypergraphic approach to algebraic graph theory}.
\newblock {\em Linear Algebra and its Applications}, 437:2262--2270, 2012.

\bibitem{Hypergraphs}
J.~Jost and R.~Mulas.
\newblock {Hypergraph Laplace operators for chemical reaction networks}.
\newblock {\em Advances in Mathematics}, 351:870--896, 2019.

\bibitem{MulasZhang}
R.~Mulas and D.~Zhang.
\newblock {Spectral theory of Laplace Operators on oriented hypergraphs}.
\newblock {\em Discrete Math.}, 344(6):112372, 2021.

\bibitem{related2018}
P.~Li and O.~Milenkovic.
\newblock Submodular hypergraphs: p-laplacians, cheeger inequalities and
  spectral clustering.
\newblock In {\em Proceedings of the 35th International Conference on Machine
  Learning, PMLR 80:3014-3023}, 2018.

\bibitem{related2018_1}
T.H.H. Chan, A.~Louis, Z.G. Tang, and C.~Zhang.
\newblock {Spectral Properties of Hypergraph Laplacian and Approximation
  Algorithms}.
\newblock {\em J. ACM}, 65(3), March 2018.

\bibitem{related3}
M.~Ikeda, A.~Miyauchi, Y.~Takai, and Y.~Yoshida.
\newblock {Finding Cheeger Cuts in Hypergraphs via Heat Equation}.
\newblock arXiv:1809.04396.

\bibitem{related2020}
A.~Banerjee.
\newblock On the spectrum of hypergraphs.
\newblock {\em Linear Algebra and its Applications}, 2020.

\bibitem{Sharp}
R.~Mulas.
\newblock {Sharp bounds for the largest eigenvalue of the normalized hypergraph
  Laplace Operator}.
\newblock {\em Math. Notes}, 2021.
\newblock To appear.

\bibitem{AndreottiMulas}
E.~Andreotti and R.~Mulas.
\newblock {Spectra of Signless Normalized Laplace Operators for Hypergraphs}.
\newblock arXiv:2005.14484.

\end{thebibliography}
				\end{document}